\documentclass[conference]{IEEEtran}

\usepackage{graphicx}
\usepackage{subfig}

\usepackage{color} 
\usepackage{url}  
\usepackage{amsmath}
\usepackage{amsthm}
\usepackage{amssymb}
\usepackage{makecell}

\numberwithin{equation}{section}
\newtheorem{theorem}{Theorem}[section]
\newtheorem*{theorem*}{Theorem}

\newtheorem{conjecture}[theorem]{Conjecture}
\newtheorem{corollary}[theorem]{Corollary}
\theoremstyle{definition}
\newtheorem{remark}[theorem]{Remark}

\newtheorem{example}[theorem]{Example}
\newtheorem{definition}[theorem]{Definition}

\newcommand{\R}{\mathbb{R}}
\newcommand{\Z}{\mathbb{Z}}

\newcommand{\C}{\mathbb{C}}

\newcommand{\cchar}{\mathrm{char}}

\DeclareMathOperator{\Hom}{Hom}

\DeclareMathOperator{\Frac}{Frac}

\DeclareMathOperator{\diag}{diag}
\DeclareMathOperator{\Sym}{Sym}

\IEEEoverridecommandlockouts 

\begin{document}

\title{Orbit recovery from invariants of low degree in representations of finite groups}

\author{
    \IEEEauthorblockN{Dan Edidin and Josh Katz}
    \thanks{Department of Mathematics,  University of Missouri, MO, USA. 
    The authors were supported by BSF grant no. 2020159 and NSF grant DMS2205626.}
}

\maketitle
 

\begin{abstract}
Motivated by applications to equivariant neural networks and cryo-electron microscopy we consider the problem of recovering the generic orbit in a representation of a finite group from invariants of low degree. The main result proved here is that invariants of degree at most three separate generic orbits in the regular representation of a finite group defined over any infinite field. This answers a question posed in~\cite{bandeira2023estimation}. We also discuss this problem for subregular representations of the dihedral and symmetric groups.
\end{abstract}

\section{Introduction}

Let $V$ be a representation of a group $G$. A classical problem in invariant theory is to give bounds on the number and degrees of generators for the invariant ring $k[V]^G$. 
A more recent problem, introduced 
by Derksen and Kemper~\cite{derksen2015computational}, is to find similar bounds for separating invariants - that is collections of invariants that separates all orbits~\cite{draisma2008polar, domokos2017degree}.
 Motivated by a broad range of applications, 
we consider the related problem of determining lower bounds on the degrees of invariants necessary to separate
{\em generic} orbits in a representation of a compact group~\cite{bandeira2023estimation, edidin2024orbit, blum-smith2024degree}.
This problem has two important applications.
The first is to the problem of constructing efficient equivariant neural networks. 
The second is to the multi-reference alignment (MRA) problem, which arises in cryo-electron miscroscopy.  In both applications it is important to  separate generic orbits with invariants of the lowest possible degree. 
Remarkably, the degrees of invariants needed to separate generic
orbits can be significantly lower than the degrees of invariants
needed to separate all orbits. For example if $G = \Z_p$ with $p$
prime acting by cyclic shift on $\C^p$ then invariants of degree $p$
are needed to separate all orbits~\cite{domokos2017degree}, while it
is known that invariants of degree three separate generic orbits
\cite{bandeira2023estimation}.

Previous papers studied the generic orbit separation problem for band-limited functions in the regular representation $L^2(G)$ when $G$ is a connected
compact Lie group \cite{edidin2024orbit}, and finite abelian groups~\cite{bandeira2023estimation, blum-smith2024degree}. 
The focus of this paper is on representations of non-abelian finite groups.
The main result proved here states that for the regular representation of a finite group defined over an infinite field invariants of degree at most three separate generic orbits. This answers a problem posed in~\cite[Remark 4.3]{bandeira2023estimation} and extends Theorem 4.1 of loc. cit. to arbitrary finite groups. We also discuss results and questions on orbit separation for subregular representations
of dihedral and symmetric groups.

\textbf{Connection to deep learning.}
In machine learning it is desirable to build neural networks that reflect the intrinsic symmetry of the data such as point clouds. If
$G$ is the symmetry group of the data then we want to build the network from $G$-equivariant functions and a basic model of an equivariant neural network
\cite{cohen2016group,lim2022equivariant} is a sequence of maps
\[\R^{n_0} \stackrel{A_1} \to \R^{n_1} \stackrel{\sigma_{b_1}} \to
\R^{n_2} \ldots \stackrel{\sigma_{b_{k-1}}} \to \R^{n_{k-1}}
\stackrel{A_k} \to \R^{n_k}\] where each $\R^{n_i}$ is a
representation of $G$, the $A_i$ are $G$-equivariant linear
transformations and the $\sigma_{b_i}$ are non-linear maps.
A difficulty with this model is that there are in general relatively few $G$-equivariant linear maps between representations, so an equivariant network built this way may not be sufficiently expressive. One approach to circumvent this problem is to use invariant polynomial maps rather than linear maps~\cite{kondor2018clebschgordan, dym2020universality,blumsmith2023machine}. However, because the cost of computing invariants grows exponentially in the degree, it is imperative to classify representations for which invariants of low degree separate almost all orbits.

\textbf{Connection to multi-reference alignment and cryo-EM.}
Given a representation $V$ of a compact group $G$, the muti-reference
alignment (MRA) problem is that of recovering the orbit of
a signal vector 
$x\in V$ from its noisy translates by unknown random group elements
\begin{align} \label{eq.mra}
y_i=g_i \cdot x+\epsilon_i
\end{align}
and the $\epsilon_i$ are taken from a Gaussian distribution $N(0,\sigma^2I)$ which is independent of the group element $g_i$. 
This  model was studied in~\cite{bandeira2014multireference}
for the regular representation of the cyclic group $\Z_n$ and in~\cite{bandeira2020non} it
was proposed as an abstract version of the cryo-EM signal reconstruction problem.  
The MRA problem for the regular representation of $\Z_n$ has been extensively studied in recent years, with both uniform an non-uniform distributions being considered~\cite{perry2019sample, bendory2017bispectrum,abbe2018multireference, bandeira2020optimal}. Other models include the dihedral group with non-uniform distribution~\cite{bendory2022dihedral}
and rotation groups acting on spaces of band-limited functions in $\R^2$
and $\R^3$ 
with uniform distribution~\cite{bandeira2020non,ma2019heterogeneous,
janco2022accelerated,edidin2024orbit}. 

When the signal-to-noise ratio is extremely low, as is the case for cryo-EM measurements, there is no way to estimate the unknown group elements, but it can be shown that the moments of the unknown signal can be accurately approximated~\cite{abbe2018estimation}. When the distribution of random group elements is governed by a uniform distribution the moments are invariant tensors and the MRA problem reduces to the problem of recovering an orbit from its invariants. However, the sample complexity (the minimal number of measurements necessary for accurate approximation) grows
exponentially in the degrees of the moments, so to efficiently solve the MRA problem it is necessary to recover almost all signals from invariants of the lowest possible degree.

\section{Polynomial and unitary invariant tensors} \label{sec.moments}
\begin{definition}[Invariant tensor] \label{def.inv}
  If $V$ is a representation of finite group $G$ defined over a field
  $k$
then the symmetric tensor
\begin{align}
  T^G_d(x)=
  \sum_{g\in G}(g \cdot x)^{\bigotimes d}
\end{align}
is called the degree $d$ invariant tensor.
\end{definition}
If $\cchar k \nmid |G|$ then $T^G_d(x)/|G|$ is the projection of $x^{\otimes d}$
to $(\Sym^dV)^G$. This definition can be extended to any linearly reductive algebraic group over $k$ with
the Reynolds operator replacing the average over the group.

If $V$ is a complex representation of a finite group $G$ then $V$ is unitary
and we define unitary invariants as follows.
\begin{definition}[The moment tensor] \label{def.moment}
  If $V$ is a complex representation of a finite group then the $d$-th moment tensor of $x \in V$ is
  \begin{align}
    T^G_d(x)=
    \sum_{g\in G}(g \cdot x)^{\bigotimes d-1} \otimes \overline{g \cdot x}
\end{align}
\end{definition}
If we divide by $|G|$ then both Definition~\ref{def.inv} and Definition~\ref{def.moment}
can be extended to any unitary representation of a compact group where the sum is replaced by the integral over the group with respect to the Haar measure.

By definition, both the moment and invariant tensors are invariants; i.e. for any $h \in G$,  $T_d^G(h \cdot x) = T_d^G(x)$ and 
$M_d^G(h\cdot x) = M_d^G(x)$. In particular
we may view the $d$-th moment tensor as a map 
$V \to (V^{\otimes d-1} \otimes V^*)^G = \Hom_G(V^{\otimes d-1} ,V)$ and the $d$-th invariant
tensor as a degree-$d$ polynomial map $V \to (V^{\otimes d})^G$.
When the group is compact or linearly reductive, then the coefficients of $T_d^G(x)$ are exactly the 
$G$-invariant polynomials of degree $d$ in $x$. The coefficients of $M_d^G(x)$ are also $G$-invariant functions but are not polynomials due to the presence of the conjugated terms. We will refer to these non-polynomial invariants occurring in the moment tensor as 
{\em unitary invariants}.
\begin{definition} If $V$ is a representation of a finite group (over an arbitrary field) or a compact Lie group then we say that polynomial invariants of degree at most $d$ separate generic
  orbits if there is a non-empty $G$-invariant Zariski open set $U\subset V$ such that for $x \in U$ the orbit of $x$ is uniquely determined by the invariant
  tensors $T^G_1(x), \ldots T^G_d(x)$.
\end{definition}
\begin{remark} As discussed in~\cite{bandeira2023estimation} if the field $k$ is algebraically closed then the condition
  that invariants of degree at most $d$ separate generic orbits is equivalent
  to the condition that $\Frac (k[V]^G) = \Frac k[V^G_{\leq d}]$ where $V_{\leq d}^G \subset k[V]^G$ is the finite dimensional subspace of invariant polynomials of degree at most $d$. The analogous statement holds for complex representations of compact Lie groups. In this case the invariant ring $\C[V]^G$ is the same
  as the invariant ring of the corresponding complex algebraic group $G_\C$. 
\end{remark}  
The following examples illustrate the difference between polynomial and unitary invariants and their ability to separate generic orbits.

\begin{example} \label{ex.goodinvariants}
Let $\mathbb{C}^n$ be the standard representation of the cyclic group $\Z_n$ represented in the Fourier basis. In other words an element
$\ell \in \Z_n$ acts by the rule 
$$\ell \cdot (x_0, \ldots , x_{n-1}) = (x_0, e^{2\pi \iota \ell}x_1,
\ldots e^{2\pi \iota (n-1) \ell}x_{n-1})$$

Because the action is diagonalized all invariants are monomials.
There is a single unitary and polynomial invariant of degree one, namely the function $x_0$. 
The degree two polynomial invariants are $\{x_ix_{n-i}\}$ whereas the degree two unitary invariants are $\{x_i\overline{x_i}\}$. In
signal processing the set of degree two unitary invariants is called the {\em power spectrum}. 
The degree three polynomial invariants are the monomials
$x_i x_j x_{n-i-j}$ where all indices are taken modulo $n$. 
Likewise the degree-three unitary invariants are the monomials
$x_i x_j \overline{x_{i+j}}$. In signal processing the collection of degree-three unitary invariants is called the {\em bispectrum}. 

Note that if $x = (x_0, \ldots , x_n)$ is the Fourier transform of a real vector then $x_i = \overline{x_{n-i}}$ so the polynomial and unitary invariants  agree on these vectors. 

Since $\C^n$ with our chosen action of $\Z_n$ is the regular representation, we know, by Tannaka-Krein duality \cite{smach2008generalized} that
the unitary invariants of degree at most three separate generic orbits. Likewise, \cite[Theorem 4.1 ]{bandeira2023estimation} or our  Theorem~\ref{thm.jennrich} implies that the polynomial invariants of degree at most three also separate generic orbits.
\end{example}

\begin{example} \label{ex.badinvariants}
  Now let let $S^1 = U(1)$ act on $\mathbb{C}^2=\mathbb{C}_1+\mathbb{C}_2$ with weights 1 and 2 i.e. $\theta\cdot x_1=e^{i\theta}x_1$ and $\theta\cdot x_2=e^{2i\theta}x_2$. There are no non-constant polynomial invariants and hence orbit
recovery from these invariants is not possible. However, the unitary invariants $x_1\overline{x_1}, x_2\overline{x_2}, x_1^2\overline{x_2}$ separate any $S^1$ orbits satisfying  $x_1x_2\neq 0$.
\end{example}
\begin{remark}
Note that the representation in Example~\ref{ex.goodinvariants} is defined over the reals, since it is the regular representation of 
a finite group. By contrast, the representation we consider in Example~\ref{ex.badinvariants} is not defined over the reals. A natural question 
is whether for real representations of finite groups 
the separating power of algebraic and unitary invariants on complex vectors is the same. 
\end{remark}

\section{The regular representation of a finite group}\label{sec:reg}
It was established in \cite{smach2008generalized} 
that if $G$ is a compact
group and $V=L^2(G)$ is the regular representation and $f\in L^2(G)$ is a function whose Fourier coefficients are all invertible then the orbit of $f$ is determined
from the bispectrum which encodes the same information as the first three
moment tensors. The proof uses Tannaka-Krein
Duality along with other results in abstract harmonic analysis. 
When $G$ is finite the
question as to whether a similar uniqueness result holds for generic recovery from the first three 
invariant tensors was posed in~\cite[Remark 4.3]{bandeira2023estimation}
and proved for finite abelian groups using Galois theory~\cite[Theorem 4.1]{bandeira2023estimation}.

In~\cite[Theorem 4.1]{bandeira2023estimation}
the authors prove that if $G$ is an arbitrary finite group,  
Jennrich's algorithm for decomposition of real tensors implies that
the $G$-orbit of 
a real valued function $f \colon G \to \R$ can be recovered from the invariants of degree at most three - provided that the orbit of
$f$ consists of linearly independent functions.
Here we use an adaptation of Jennrich's algorithm to tensors over arbitrary
fields and prove
that if $k$ is any infinite field, with no restriction on the characteristic,
then 
polynomial invariants of degree at most three separate generic orbits
in the regular representation, and more generally any representation
where the generic orbit consists of linearly independent vectors.

\begin{theorem} \label{thm.jennrich}
  Let $V$ a representation of a finite group over defined over an infinite field
  $k$
  then the $G$-orbit of
  any vector $x \in V$ whose orbit consists of linearly independent vectors
  is uniquely determined 
from the invariant tensors $T^G_2(x)$ and $T^G_3(x)$.
\end{theorem}
\begin{proof}
  Enumerate the elements of $G$ as $g_1, \ldots , g_{n}$ where
  $g_1 = e$. If $x \in V$ set $x_i = g_i x$. By assumption on $x$, the
  vectors $x_1, \ldots , x_{n}$ are linearly independent and their span
is a $G$-invariant subspace $W \subset V$ isomorphic to the regular representation. We can recover the subspace $W$ from the 
second invariant tensor  
$$T^G_2(x) = \sum_{g \in G}  gx \otimes gx =\sum_{i=}^{n} x_i \otimes x_i$$
without a priori knowing the vector $x$ as follows. If we choose a basis for $V$ then we can represent
$T^G_2(x)$ as a symmetric matrix with respect to this basis and the range
of this matrix will be the $G$-invariant subspace $W$. (Note that we do not need $k$ to be an infinite field to use this argument.)

  
We now consider the degree-three invariant tensor
$T^G_3(x) = \sum_{i=1}^{n} x_i \otimes x_i \otimes x_i$.
Fix two linearly independent elements $a,b \in V^*$ 
and consider the contractions
$T_a = \sum_{i=1}^{n} \langle a, x_i \rangle x_i \otimes x_i$ and
$T_b = \sum_{i=1}^{n} \langle b, x_i \rangle x_i \otimes x_i$ in $\Sym^2 W
\subset \Sym^2 V$.
Choosing an ordered basis for $W$  we can 
represent $T_a$ and $T_b$ as symmetric $n \times n$-matrices.

Let $D_a=\diag{\langle a,x_i\rangle}$ and $D_b=\diag{\langle
  b,x_i\rangle}$ then $T_a=XD_aX^T$ and $T_b=XD_bX^T$ where
$X=(x_1,\ldots ,x_{n})$ is the matrix whose columns are
$\{x_i\}$. Since the $x_i$ are linearly independent the matrix $X$ is invertible. Moreover, if the field $k$ is infinite  then for almost  all choices
of dual vectors $a,b \in V^*$
$\langle a, x_i \rangle$ and $\langle b, x_i \rangle$ are all non-zero, and the products
$\langle a, x_i \rangle \langle b, x_i \rangle^{-1}$ are distinct elements
of the field $k$. In particular
the matrices $T_a$ and $T_b$ are invertible and 
the product $T_aT_b^{-1}=XD_aD_b^{-1}X^{-1}$, which
can be determined from the third invariant tensor has 
distinct eigenvalues, and hence one-dimensional eigenspaces.
To determine the $x_i$ we can use the following strategy. Let
$u$ be any eigenvector of $T_aT_b^{-1}$. Then $x = c h x$ for some
$h \in G$  and non-zero constant $c \in \C$. Then
$\{ gu\}_{g \in G} = \{ cx_i\}$.
Hence $\sum_{g \in G} gu^{\otimes 3} = c^3 T_3(x)$. Since $T_3(x)$ is
known, we determine $c^3$.
We can also
determine $c^2$ and thus $c$,  by comparing $\sum gu^{\otimes 2}$
with the known invariant tensor $T^G_2(x)$.
\end{proof}

\section{Subregular representations of finite groups}
Theorem~\ref{thm.jennrich} implies that the generic orbit in any representation that contains a copy of the regular representation
can be recovered from the first three moment or invariant tensors.
However, this condition is certainly not necessary,
and we can
ask for
non-trivial examples of subregular
representations where invariants of degree at most three separate
generic orbits.  

\textbf{The dihedral group.}
The dihedral group provides a class of examples of representations properly
contained in the 
regular representation for which invariants of degree at most three
separate generic orbits.

Let $D_{n}$ be the dihedral group of order $2n$ with generators
$r$ of order $n$ and $s$ of order two. Consider the $n$-dimensional
standard representation of $D_{n}$ where the generator $r$ acts by cyclic shifts and the generator $s$ acts by the reflection
$s(x_0, \ldots x_{n-1}) = (x_0, x_{n-1}, x_{n-2}, \ldots x_1)$.

\begin{theorem} \cite{edidin2024generic} \label{thm.dihedral}
  The first three invariant tensors separate generic complex orbits
  in the standard representation, $\mathbb{C}^n$, of $D_{n}$.
\end{theorem}

\begin{example}[The complete multiplicity-free representation of the dihedral group]
We define the complete multiplicity-free representation $V$ of finite group to be
the sum of all irreducibles taken with multiplicity one. Here we show
that for the complete multiplicity-free representation of $D_n$ generic orbits cannot be separated by invariants of degree at most three. 

The irreducible representations of $D_{n}$ are as follows.
If $n$ is even, then there are $(n/2 -2)$ two-dimensional irreducible
representations $V_1, \ldots V_{(n/2 -2)}$ where the
rotation acts on $V_\ell$ with weights $e^{\pm 2\pi \iota \ell/n}$ and
the reflection exchanges the two eigenspaces for the rotation. There
are four characters, $L_0, L_{-1}, S_0,S_{-1}$. The reflection
acts trivially on $L_0, L_{-1}$ and the rotation acts with weights 1 (i.e.  trivially)  and -1
respectively. On $S_0, S_{-1}$ the reflection acts with weight $-1$ and the rotation acts with weights $1$ and $-1$ respectively. 

If $n$ is odd then there are $(n-1)/2$ two-dimensional representations
$V_1, \ldots , V_{(n-1)/2}$ where the rotation acts on $V_\ell$ with eigenvalues
$e^{\pm 2\pi\iota \ell/n}$ and the reflection exchanges the two eigenspaces.
In addition if $n$ is odd $D_{n}$ has two characters, the trivial
character $L_0$, and the character $S_0$ where the rotation acts trivially
and the reflection acts with weight $-1$. 

The $n$-dimensional standard representation of $D_{n}$ is the sum $L_0 \oplus V_1 + \ldots V_{n/2 -2} \oplus L_{-1}$ if $n$ is even
and $L_0 \oplus V_1 \ldots \oplus V_{(n-1)/2}$ is $n$ is odd.
In particular it is multiplicity free.

 Interestingly we cannot separate the generic orbit with invariants of degree at most three on $V$ even though we can on both the standard representation and the regular representation.
  \begin{corollary} \label{cor.dihedralnegative}
  Let $V$ be the complete multiplicity-free representation of the dihedral
  group $D_{n}$. Then the generic orbit $x \in V$ can only be recovered up to a list of size two for the invariants of degree at most three.
\end{corollary}  

\begin{proof} If $n$ is odd we can write an element of $V$ as
$(x_1, \ldots, x_n, s_0)$ where $(x_1, \ldots, x_n)$ are coordinates for the standard reprsentation and $s_0$ is coordinate for $S_0$.  For a generic choice of
$(x_1, \ldots , x_n)$ and non-zero $s_0$ the vectors
$(x_1,\ldots, x_n, s_0)$ and $(x_1, \ldots, x_n, -s_0)$ lie in different $D_n$-orbits but their invariants of degree three or
less are equal. In order to separate the orbits we need invariants
of degree four. Precisely we need invariants of the form
$s_0p(x_1,\ldots, x_n)$ where $p(x_1,\ldots ,x_n)$ is invariant
under rotations and the reflection acts by $-1$. However lowest degree of such a semi-invariant polynomial
$p(x_1, \ldots , x_n)$ is three.

If $n$ is even we can write an element of $V$
  as $(x_1,\ldots, x_n, s_0,s_{-1})$ and for a generic choice
  of $(x_1, \ldots, x_n)$ and $s_0,s_1$ both non-zero, the vectors $(x_1, \ldots , x_n, s_0,s_{-1})$
  and $(x_1, \ldots , x_n,  -s_0, -s_{-1})$ have the same invariants of degree three or less but do not lie in the same $D_n$ orbit. 
\end{proof}
\end{example}

\textbf{The symmetric group.} 
Deep learning of objects such as sets or graphs which do not come with intrinsic ordering requires the construction of invariants which separate  orbits under the symmetric group~\cite{balan2022permutation,segol2020universal, tabaghi2023universal, dym2022low}. Motivated by these questions, we consider the permutation action of $S_n$ on the space $\mathbb{C}^n\otimes (\mathbb{C}^d)^*=\mathbb{C}^{n\times d}$ of $n\times d$ matrices and pose the problem of determining bounds on the multiplicity $d$ which ensure that a set of low degree invariants
can separate generic orbits.

Recall that $\mathbb{C}^n$ splits as an $S_n$ representation into
$\mathbb{C}^n=V_0+ V$ where $V_0$ is the trivial representation generated by
$(1,\ldots , 1)$ and $V$ is its orthogonal complement which
is the standard $(n-1)$-dimensional representation of $S_n$.
In particular, the space of matrices $(\mathbb{C}^n)^d=V^d+V_0^d$ is not a subrepresentation of the regular representation when $d>1$,
since the trivial representation only occurs with multiplicity one in the regular representation. 
However, we lose no important information by removing the trivial summand $V_0^d$ and working with just $V^d$ which is a subregular representation for $d\leq n-1$. Hence the space of matrices $\mathbb{C}^{n\times d}$ can appropriately be considered as a subregular representation for $d\leq n-1$.


  The ring of invariants is generated by {\em multisymmetric polynomials} of degrees $1$ through $n$ \cite{schlafli1852resultante,rydh2007generators}
  and when $d > 1$ they are no longer algebraically independent so it is possible for invariants of low degree to separate generic orbits.
The multisymmetric polynomials can be enumerated as follows:
If $\{x_{i,j}\}$ for $1\leq i\leq n, 1\leq j \leq d$ is a dual basis for 
$(\mathbb{C}^{n})^d$
then the sum
$\sum_{\sigma\in S_n, 1\leq i_1,...,i_k\leq d} \sigma(x_{1,i_1}...x_{1,i_k})=x_{1,i_1}...x_{1,i_k}+x_{2,i_1}...x_{2,i_k}+...+x_{n,i_1}...x_{n,i_k}$ is an $S_n$-invariant polynomial of degree $k$ (note that the $i_1,...,i_k$ do not need to be distinct). We refer to invariants of this form as multisymmetric power sum polynomials.

There is a one-to-one correspondence between the power sum multisymmetric polynomials and the monomials of the form $x_{1,i_1}...x_{1,i_k}$ in the variables $\{x_{1,i}\}_{1\leq i\leq d}$. Therefore, the number of distinct power sum multisymmetric polynomials of degree $k$
is $\binom{d+k-1}{k}$.
In particular there are  
$\binom{d+2}{3}+ \binom{d+1}{2} + \binom{d}{1} = \frac{1}{6}(d^3 + 6d^2 + 11d)$ such invariants of degree at most three.

Recall the following definition given in~\cite{bandeira2023estimation}.
\begin{definition}
  Let $S \subset k[V]^G$ be a set of polynomial invariants for the action
  of a finite group.
  We say that $S$ list resolves generic orbits in $V$ if it contains a transcendence basis for $\Frac( k[V]^G)$.
\end{definition}

We give numerical
evidence that, for $d$ sufficiently large, the
invariants of degree at most three list resolve the generic orbit in
$V=\mathbb{C}^{n\times d}$.  A necessary condition for the invariants
of degree at most three to contain a transcendence basis is that
$\frac{1}{6}(d^3 + 6d^2 + 11d) \geq nd$.  Asymptotically this forces
$d$ to be at least $\sim \sqrt{n}$. To produce the table below we generated
all of the multisymmetric power sum polynomials of degree at most
three for different values of $n$ and $d \leq n-1$ and checked whether
or not they contain a transcendence basis using the Jacobian criterion as
developed in~\cite[Proposition 3.25]{bandeira2023estimation} and
implemented in Mathematica.

\begin{table}[h!]
\centering
\begin{tabular}{|c|c|c|}
\hline
\textbf{Group} & \textbf{Representation} & \parbox[c]{4.5cm}{\centering \textbf{Invariants of degree at most 3}\\ \textbf{contains a transcendence basis?}} \\ \hline
\(S_4\) & $\mathbb{C}^{4}$ & No \\
\hline
\(S_4\) & $\mathbb{C}^{4\times 2}$ & Yes \\ \hline
\(S_5\) & $\mathbb{C}^{5}$ & No \\ \hline
\(S_5\) & $\mathbb{C}^{5\times 2}$ & No \\ \hline
\(S_5\) & $\mathbb{C}^{5\times 3}$ & Yes \\ \hline
\(S_6\) & $\mathbb{C}^{6}$ & No \\ \hline
\(S_6\) & $\mathbb{C}^{6\times 2}$ & No \\ \hline
\(S_6\) & $\mathbb{C}^{6\times 3}$ & Yes \\ \hline
\end{tabular}

\caption{Invariants and transcendence bases for symmetric groups}
\label{table:degree3invariants}
\end{table}

These computations lead us to make the following conjecture.

\begin{conjecture}
The invariants of degree at most three list resolve the generic orbit in $(\mathbb{C}^n)^d$ if and only if\\ $\frac{1}{6}(d^3 + 6d^2 + 11d) \geq nd$.
\end{conjecture}

\bibliographystyle{plain}

\end{document}